\documentclass{kms-b}

\newcommand{\publname}{}
\issueinfo{}
  {}
  {}
  {}
\pagespan{1}{}
\copyrightinfo{}
  {Korean Mathematical Society}

\usepackage{graphicx}
\allowdisplaybreaks

\theoremstyle{plain}
\newtheorem{theorem}{Theorem}[section]
\newtheorem{proposition}[theorem]{Proposition}
\newtheorem{lemma}[theorem]{Lemma}

\theoremstyle{definition}

\theoremstyle{remark}

\begin{document}

\title[]
{Existence and uniqueness of solutions to Bogomol'nyi equations on graphs}

\author[Y. Hu]{Yuanyang Hu}
\address{Yuanyang Hu \\ School of Mathematics and Statistics \\ Henan University \\  Henan 475004, P. R. China.}
\email{yuanyhu@henu.edu.cn}

\thanks{This work was financially supported by NSFC Grant 12201184 and China Postdoctoral Science Foundation Grant 2022M711045.}


\begin{abstract}
Let $G=(V,E)$ be a connected finite graph.
We study the Bogomol'nyi equation 
\begin{equation*}
	\Delta u= \mathrm{e}^{u}-1 +4 \pi \sum_{s=1}^{k} n_s \delta_{z_{s}} \quad \text { on } \quad G,
\end{equation*}
where $z_1, z_2,\dots, z_k$ are arbitrarily chosen distinct vertices on the graph, $n_j$ is a positive integer, $j=1,2,\cdots, k$  and $\delta_{z_{s}}$ is the Dirac mass at $z_s$. We obtain a necessary and sufficient condition for the existence and uniqueness of solutions to the Bogomol'nyi equation.
\end{abstract}

\maketitle

\section{Introduction}

Magnetic vortices play important roles in many areas of theoretical physics including condensed-matter physics, cosmology, superconductivity theory, optics, electroweak theory, and quantum Hall effect. In \cite{JaffeTaubes}, Jaffe and Taubes established the existence of multivortex solutions to the Bogomol'nyi equations. Later, Wang and Yang \cite{WangYangSiam} established a sufficient and necessary condition for the existence of multivortex solutions of the Bogomol'nyi system. Recently, $(2+1)-$dimensional Chern-Simons gauge theory and generalized Abelian Higgs model have attracted extensive attention. The topological, non-topological and doubly periodic multivortices to the self-dual Chern-Simons model and the Abelian Higgs model were established over the past two decades; see, for example, \cite{ ChenHanLozaJDE, ChaImanuvilov, Hannonlinearity, Tarantello, TchraYang, YangchenHPC} and the references therein.

The discrete version of partial differential equations has found many application in many fields, see \cite{Elmoataz, Nakanishi}, for example. Analysis and partial differential equations on locally finite graphs have attracted a considerable amount of attention recently. For the aspect of Kazdan-Warner equations, see \cite{Gekazdan, GrigoryanLinYangCVPDE, GehuabinJiangwenfengJKMS}, for example. For the Yamabe type equations, we refer the readers to \cite{GehuabinYamabe, GehuabinJiangwenfeng, GrigoryanLinYangJDE}. For the counterpart of epidemic models, see \cite{AllenBolkerSIAM, TianLiuRuan} and the references therein. For the porous medium equations, see \cite{BianchiSettiCVpde, ErbarMaasdcds}. For the Fujita type heat equation, see \cite{Chungsy, LinWuCVPDE, WuyingtingRACSAM}. On the Chern-Simons equations on finite graph, see \cite{HuangLinYauCMP, HouSunCVPDEs, HuangWangYangJFA, LuZhongair}.

Inspired by these works \cite{HuangLinYauCMP, HouSunCVPDEs, HuangWangYangJFA, LuZhongair}, we study the Bogomol'nyi equation  
\begin{equation}\label{E}
	\Delta u= \mathrm{e}^{u}-1 +4 \pi \sum_{s=1}^{k} n_s \delta_{z_{s}} ,
\end{equation} on $G$, where $G=(V,E)$ is a connected finite graph, $k\ge 1$ is an integer, $n_{s}>0$ is an integer, $s=1,\cdots,k$ , $\delta_{z_{s}}$ satisfies 
\begin{equation}
	\delta_{z_{s}}(x)= \begin{cases} \frac{1}{\mu(z_{s})}, & x= z_{s} , \\ 0, & \text { otherwise }, \end{cases}
\end{equation}
and $z_1, z_2,\dots, z_k$ are arbitrarily chosen distinct vertices on the graph $G$. 

It is meaningful to study the partial differential equations on graphs, since the graph is a generation of the lattice $\mathbb{Z}^{d}$. We hope our results can be applied to numerical analysis.  

The paper is organized as follows. In Section 2, we present some basic results that will be used frequently in the following sections and state the main result. In Section 3, we give the proof of the main result.

\section{Preliminary results}
Let $G=(V,E)$ be a finite graph, where  $V$ denotes the vertex set and $E$ denotes the edge set. Throughout this paper, all graphs are supposed to be connected. We write $y\sim x$ if $xy\in E$. 
Let $\omega: V\times V\to [0,+\infty)$ be an edge weight function that satisfies $\omega_{xy}=\omega_{yx}$ for all $x,y\in V$ and $w_{xy}>0$ if and only if $x\sim y$.  Let $\mu: V \to (0,+\infty)$ be a finite measure, and 
$$|V|= \text{Vol}(V):=\sum \limits_{x \in V} \mu(x)$$ be the volume of $V$. Define $V^{\mathbb{R}}:=\{u|u~\text{is a real function}: V\to \mathbb{R}\}$. For any function $u\in V^{\mathbb{R}}$, the Laplacian of $u$ is defined by 
\begin{equation}\label{1}
	\Delta u(x)=\frac{1}{\mu(x)} \sum_{y \sim x} w_{y x}(u(y)-u(x)),
\end{equation}
where $y \sim x$ means $xy \in E$. The gradient form of $u$ reads 
\begin{equation}\label{g}
	\Gamma(u, v)(x)=\frac{1}{2 \mu(x)} \sum_{y \sim x} w_{x y}(u(y)-u(x))(v(y)-v(x)).
\end{equation}
We denote the length of the gradient of $u$ by
\begin{equation}\label{gra}
	|\nabla u|(x)=\sqrt{\Gamma(u,u)(x)}=\left(\frac{1}{2 \mu(x)} \sum_{y \sim x} w_{x y}(u(y)-u(x))^{2}\right)^{1 / 2}.
\end{equation}
Denote, for any function $
u: V \rightarrow \mathbb{R}
$, the integral of $u$ on $V$ by $$\int \limits_{V} u d \mu=\sum\limits_{x \in V} \mu(x) u(x).$$ For $p \ge 1$, denote $$|| u ||_{p}:=(\int \limits_{V} |u|^{p} d \mu)^{\frac{1}{p}}$$
and $$L^{p}(V)=\{f\in V^{\mathbb{R}}: ||f||_{p}<\infty\}.$$ 
As in \cite{GrigoryanLinYangCVPDE}, we define a Sobolev space and a norm by 
\begin{equation*}
H^{1}(V)=W^{1,2}(V)=\left\{u: V \rightarrow \mathbb{R}: \int \limits_{V} \left(|\nabla u|^{2}+u^{2}\right) d \mu<+\infty\right\},
\end{equation*}
and \begin{equation*}
	\|u\|_{H^{1}(V)}=	\|u\|_{W^{1,2}(V)}=\left(\int \limits_{V}\left(|\nabla u|^{2}+u^{2}\right) d \mu\right)^{1 / 2}.
\end{equation*}

The following Sobolev embedding, Maximum principle and Poincar\'{e} inequality will be used later in the paper.
\begin{lemma}\label{21}
	{\rm (\cite[Lemma 5]{GrigoryanLinYangCVPDE})} Let $G=(V,E)$ be a finite graph. The Sobolev space $W^{1,2}(V)$ is precompact. Namely, if ${u_j}$ is bounded in $W^{1,2}(V)$, then there exists some $u \in W^{1,2}(V)$ such that up to a subsequence, $u_j \to u$ in $W^{1,2}(V)$.
\end{lemma}

\begin{lemma}\label{22}
	{\rm (\cite[Lemma 4.1]{HuangLinYauCMP})} Let $G=(V,E)$ be a graph, where $V$ is a finite set, and $K> 0$ is constant. Suppose a real function $u(x): V\to \mathbb{R}$ satisfies $(\Delta-K)u(x) \ge 0$ for all $x\in V$, then $u(x) \le 0$ for all $x \in V $.
\end{lemma}

\begin{lemma}\label{23}
	{\rm (\cite[Lemma 6]{GrigoryanLinYangCVPDE})}	Let $G = (V, E)$ be a finite graph. For all functions $u : V \to \mathbb{R}$ with $\int \limits_{V} u d\mu = 0$, there 
	exists some constant $C$ depending only on $G$ such that $$\int \limits_{V} u^2 d\mu \le C \int \limits_{V} |\nabla u|^2 d\mu.$$
\end{lemma} 

Let $f\in {V}^{\mathbb{R}}$, we are going to consider the equation 
\begin{equation}\label{6.}
	-\Delta u=f 
\end{equation}
on $G$.
\begin{lemma}\label{2.4}
Assume $G=(V,E)$ is a graph. Then there exists a solution to \eqref{6.} iff $\int_{V} fd\mu=0$. Furthermore, the solution is unique up to a constant.
\end{lemma}
\begin{proof}
We first prove sufficiency. If $f\equiv 0$, then by the definition of $\Delta$ we know that $u$ is a constant. Next, we suppose that $f\not\equiv 0$ and that $\int_{V} fd\mu=0$. Define $$I(u)=\int \limits_{V} |\nabla u|^2 d\mu$$ and $\alpha:=\inf\limits_{u \in A} I(u),$ where 
	\begin{equation*}
		A=\left\{u \in H^{1}(V) \mid \int_V u d \mu=0 \text { and } \int_V u f d\mu=1\right\} \text {. }
	\end{equation*}
Clearly, $0\le \alpha\in \mathbb{R}$, and $0 \le \alpha \le I\left(\frac{f}{\|f\|_2^2}\right)$. Let $\{u_{k}\}_{k=1}^{\infty}$ be a minimizing sequence in $A$ satisfying $I(u_{k})\to \alpha$. In view of $I(u_{k})=||\nabla u_{k}||_{2}^{2}$, we see that $\{|\nabla u_{k}|\}_{k=1}^{\infty}$ is bounded in $L^{2}(V)$.
Furthermore, since $\int_{V}u_{k}d\mu=0$, by Lemma \ref{23}, we see that 
\begin{equation*}
	\int_V u_k^2 d\mu \leq C \int_V\left|\nabla u_k\right|^2 d\mu.
\end{equation*}
Thus $\{u_{k}\}_{k=1}^{\infty}$ is bounded in $H^{1}(V)$. By Lemma \ref{21}, we can find a sequence $\{u_{i}\}$ of $\{u_{k}\}$ and $u_0\in H^{1}(V)$ so that $u_{i}\to u_{0}$ in $H^{1}(V)$. It follows that $\lim\limits_{i \rightarrow+\infty} u_i(x)=u_0(x)$ for all $x\in V$. Thus we obtain
\begin{equation*}
	\lim _{i \rightarrow \infty} \int_V\left|\nabla u_i\right|^2 d\mu=\int_V\left|\nabla u_0\right|^2 d \mu=\alpha,
\end{equation*} 
\begin{equation*}
	\lim _{i \rightarrow \infty} \int_V fu_i d\mu=\int_V fu_0 d\mu=1,
\end{equation*}
and
\begin{equation*}
	\lim _{i \rightarrow \infty} \int_V u_i d\mu=\int_V u_0 d\mu=0.
\end{equation*}
Hence, $u_0$ is the minimizer of the variational problem $\inf\limits_{u\in A}I(u)$. 

We next show that there exist constants $\beta$ and $\gamma$ such that 
\begin{equation}\label{7.7}
	\int_V \Gamma (u_0,\psi) d\mu=-\beta \int_V f \psi d\mu-\gamma \int_{V}\psi d\mu \text { for all } \psi \in H^{1}(V) \text {.}
\end{equation}
by the method of Lagrange multiplies.
Define
\begin{equation*}
L(t, \beta, \gamma)
		=\frac{1}{2} \int_V\left|\nabla\left(u_0+t \psi\right)\right|^2 d\mu+\gamma \int_V\left(u_0+t \psi\right) d\mu+\beta\left(\int_V\left(u_0+t \psi\right) f d\mu-1\right),
\end{equation*}
where $\beta, \gamma\in \mathbb{R}$ and $\psi\in H^{1}(V)$.
Then we get 
\begin{equation}\label{7.}
	\left.\frac{\partial L}{\partial \beta}\right|_{t=0}=0,\left.\quad \frac{\partial L}{\partial \gamma}\right|_{t=0}=0 \text { and }\left.\frac{\partial L}{\partial t}\right|_{t=0}=0 \text {. }
\end{equation}
By \eqref{g} and \eqref{gra}, we see that 
\begin{equation*}
	\begin{aligned}
	&\quad  \left.\frac{d}{d t} \Gamma\left(u_0+t \psi, u_0+t \psi\right)(x)\right|_{t=0}\\ & \left.=\frac{d}{d t}\left\{\frac{1}{2 \mu(x)} \sum_{y \sim x} \omega_{x y}\left[ \left(u_0+t \psi\right)(y)-\left(  u_0+t \psi\right)(x)\right]^2\right\}\right|_{t=0} \\
		& =\left.\sum_{y \sim x}\frac{\omega_{x y}}{2 \mu(x)}  2\left[\left(u_0+t \psi\right)(y)-\left(u_0+t \psi\right)(x)\right][\psi(y)-\psi(x)]\right|_{t=0} \\
		& =\frac{1}{2 \mu(x)} \sum_{y \sim x} \omega_{x y} 2\left[u_0(y)-u_0(x)\right][\psi(y)-\psi(x)] \\
		& =2 \Gamma\left(u_0, \psi\right),
	\end{aligned}
\end{equation*}
and hence that  
\begin{equation*}
	\left.\frac{\partial L}{\partial t}\right|_{t=0}=\int_V \Gamma\left(u_0, \psi\right) d \mu+\gamma \int_V \psi d \mu+\beta \int_V \psi f d\mu=0
\end{equation*} by \eqref{7.} and the definition of $L$.
Taking $\psi\equiv 1$. Since 
$$\int_V \Gamma (u_0,1) d\mu=0$$ and $\int_{V} f d\mu=0$, we get $\gamma=0$. Choosing $\psi=u_{0}$ then we get $\beta=-\alpha$. Hence, 
\begin{equation}\label{2.2}
	\int_V \Gamma (u_0,\psi) d\mu=\alpha \int_V f \psi d\mu~\text{ for all }~\psi\in H^{1}(V).
\end{equation}
Since $\int_{V}u_0 fd\mu=1$, $u_0$ is not a constant. We claim that $\alpha>0$. Otherwise $\alpha=0$. Then $\int_{V}|\nabla u_0|^{2} d\mu=0$, which implies that $|\nabla u_0|\equiv 0$. By \eqref{gra}, we deduce that $u\equiv C_2$, where  $C_2>0$ is a constant. This is a contradiction. Thus, $\alpha=I(u_0)>0$. Letting $$\tilde{u}_{0}=\frac{u_0}{\alpha}.$$ Then it follows from \eqref{2.2} that
\begin{equation*}
	\int_V \Gamma (\tilde{u}_0,\psi)  d\mu=\int_V f \psi d\mu \text { for all } \psi\in H^{1}(V).
\end{equation*}
Integrating by parts, we have 
\begin{equation*}
	-\int_V \psi \Delta \tilde{u}_0 d\mu=\int_V f \psi d\mu \text { for all } \psi \in H^{1}(V).
\end{equation*} For any $x_0\in V$, taking 
\begin{equation*}
	\psi(x)= \begin{cases}1, & x=x_0,\\
		 0, & \text{otherwise.}\end{cases}
\end{equation*}
Then we deduce that $-\Delta\tilde{u}_{0}(x_0)=f(x_0)$. Therefore, $\tilde{u}_{0}$ is a solution to \eqref{1}.

 We next prove the necessity. Suppose $-\Delta u=f$. Since $\int_{V}\Delta ud\mu=0$, we know that $\int_{V}fd\mu=0$.

	If $f\equiv 0$, then $u$ is a constant by the definition of $\Delta$. Thus, we know that the solution to \eqref{6.} is unique up to a constant.
	
We now complete the proof.
\end{proof}

We state our main result as follows.
\begin{theorem}\label{t1}
	The equation \eqref{E} admits a unique solution, if and only if 
	\begin{equation}
		n_1 + n_2 + \cdots + n_k=: N <\frac{|V|}{4 \pi} .
	\end{equation}
\end{theorem}

\section{The proof of Theorem \ref{t1}}
Throughout this section, we assume that $N=\sum\limits_{s=1}^{k} n_s$ and that $f$ is a function on $V$ satisfying $$\int_{V} f d \mu=1.$$ Since 
\begin{equation*}
	\int_{V} -4\pi Nf+4 \pi \sum_{j=1}^{k} n_j  \delta_{z_{j}} d\mu=-4\pi N+4\pi N=0,
\end{equation*}
by Lemma \ref{2.4}, we know that there exists a unique solution to the Poisson equation
\begin{equation*}
	\Delta u_{0}=-4 \pi N f +4 \pi \sum_{j=1}^{k} n_j  \delta_{z_{j}},
\end{equation*}
in the sense of differing by a constant. Assume $u$ is a solution of \eqref{E}, let $v:=u-u_0$, then $v$ satisfies 
\begin{equation*}\label{3}
	\Delta v=e^{v+u_0 }-1+ 4 \pi N f .
\end{equation*}

Define an operator $P:=\Delta-e^{u_0}-1: W^{1,2}(V) \to L^{2}(V)$, then we have the following proposition.
\begin{proposition} \label{p1}
	$P$ is bijective.
\end{proposition}
\begin{proof}
	For any $u,v \in H^{1}(V)$, define 
	$$B(u,v):=\int\limits_{V} \Gamma (u,v)+(e^{u_0} +1) uv d\mu.$$
	By Cauchy Schwartz inequality and H$\ddot{\text{o}}$lder inequality, we deduce that
	\begin{equation*}
		\begin{aligned}
			|B(u, v)| 	& \leq \int_{V} |\Gamma(u, v)|
			+\left(e^{u_{0}}+1\right)|u||v| d \mu \\
			& \leq \int_{V} \sum_{y \sim x} \frac{w_{x y}}{2 \mu(x)}|u(y)-u(x)||v(y)-v(x)| d \mu\\
			&~~~~+\max _{V}\left(e^{u_{0}}+1\right)\left(\int_{V} u^{2} d \mu \right)^{\frac{1}{2}}\left(\int_{V} v^{2} d \mu\right)^{\frac{1}{2}} \\
			& \leq \int_{V}\left(\sum_{y \sim x} \frac{w_{x y}}{2 \mu(x)}(u(y)-u(x))^{2}\right)^{\frac{1}{2}}\left(\sum_{y \sim x} \frac{w_{x y}}{2 \mu(x)}(v(y)-v(x))^{2}\right)^{\frac{1}{2}} d\mu\\
			&~~~~~+C_1 ||u||_{2}||v||_{2} \\
			&=\int_{V}[\Gamma(u, u)]^{\frac{1}{2}}[\Gamma(v, v)]^{\frac{1}{2}} d \mu+C_1 ||u||_{2}||v||_{2} \\
			&=\int_{V} | \nabla u|| \nabla v|d \mu+C_1 || u ||_{2} ||v||_{2} \\
			& \leq\left(\int_{V}|\nabla u|^{2} d \mu\right)^{\frac{1}{2}}\left(\int_{V}|\nabla v|^{2} d \mu \right)^{\frac{1}{2}}+C_{1}||u||_{2}||v||_{2},
		\end{aligned}
	\end{equation*}
	where $C_1=\max\limits_{V} (e^{u_0}+1)$.
	By \eqref{g} , we have 
	\begin{equation}\label{6}
		B(u,u)\ge \int\limits_{V} |\nabla u|^2 + \min\limits_{V}(e^{u_0}+ 1) u^2 d\mu.
	\end{equation}
	Therefore, we can find constants $C_1, C_2>0$ such that 
	\begin{equation}
		|B(u,v)| \le C_1||u||_{H^{1}(V)} ||v||_{H^{1}(V)},
	\end{equation}
	and
	\begin{equation}
		B(u,u) \ge C_2||u||^{2}_{H^{1}(V)}.
	\end{equation}
	
	It is easy to check that $B:~H^{1}(V)\times H^{1}(V)\to \mathbb{R}$ is a bilinear mapping.
	Thus, by the Lax-Milgram Theorem, for any function $g$ on $V$, there exists a unique element $u\in H^{1}(V)$ such that
	\begin{equation}\label{5}
		B(u,v)=-\int\limits_{V} gv d \mu, 
	\end{equation}
	for any $v\in H^{1}(V)$. Since \eqref{5} is equivalent to $Pu=g$, we see that $P:H^{1}(V)\to L^{2}(V)$ is bijective.
	
 We next prove necessity.  Suppose $\Delta u=f$. Since $\int_{V}\Delta u d\mu=0$, we know that $\int_{V}f d\mu=0$. 

We now complete the proof.
\end{proof}
By Proposition \ref{p1}, we can define the inverse operator of $P$ by $P^{-1}$. Furthermore, we have the following proposition. 
\begin{proposition}\label{p2}
	$P^{-1}: L^2(V) \to H^1(V)$ is compact.
\end{proposition}
\begin{proof}
	For any $g \in L^{2}(V)$, by Proposition \ref{31}, there exists $u\in H^{1}(V)$ such that $Pu=g$, which is equivalent to
	\begin{equation*}
		B(u,v)=-\int\limits_{V} gv d \mu, 
	\end{equation*} for any $v \in H^{1}(V)$. By Cauchy inequality with $\epsilon$, ($\epsilon>0$), we see that
	\begin{equation*}
		B(u,u)=-\int\limits_{ V} gu d\mu \le \frac{1}{4 \epsilon}\int\limits_{ V} g^2 d\mu+ \epsilon \int\limits_{ V} u^2 d\mu.
	\end{equation*}
	Thus, from \eqref{6}, we conclude that
	\begin{equation}\label{7}
		\int\limits_{ V} |\nabla u|^2+C_2 u^2 d \mu \le   \frac{1}{4 \epsilon}\int\limits_{ V} g^2 d\mu+ \epsilon \int\limits_{ V} u^2 d\mu.
	\end{equation}
	Taking $\epsilon=\frac{C_2}{2}$ in \eqref{7}, we deduce that 
	\begin{equation*}
		\int\limits_{ V} |\nabla u|^2+\frac{C_2}{2}  u^2 d \mu \le   \frac{1}{2 C_2}\int\limits_{ V} g^2 d\mu.
	\end{equation*}
	Therefore, there exists $C_3>0$ such that \begin{equation*}
		||u||^{2}_{H^{1}(V)} \le C_3 \int\limits_{V} g^{2} d\mu.
	\end{equation*}
	Thus, by Lemma \ref{21}, we know that $P^{-1}: L^2(V) \to H^1(V)$ is compact. 
\end{proof}

Next, we give a necessary condition for equation \eqref{3} to have a solution.
\begin{lemma}\label{31}
	If the equation \eqref{3} admits a solution, then $4\pi N< |V|.$
\end{lemma}
\begin{proof}
	Assume that $v$ is a solution of the equation \eqref{3}, then \begin{equation}
		0=\int\limits_{ V} \Delta v d \mu=\int\limits_{V} e^{v+u_0}-1+4\pi Nf d\mu=\int\limits_{ V}e^{v+u_0 } d\mu -|V| +4\pi N.
	\end{equation}
	Thus, we have $4\pi N < |V|$.
\end{proof}

The following lemma gives the uniqueness of solutions of the equation \eqref{3}.
\begin{lemma}\label{32}
	There exists at most one solution of \eqref{3}.
\end{lemma}
\begin{proof}
	If $u$ and $v$ both satisfy equation \eqref{3}, by mean value theorem, there exists $\xi$ such that
	\begin{equation*}
		\Delta u- \Delta v =e^{u+u_0}-e^{v+u_0}=e^{\xi+u_0}(u-v).
	\end{equation*}
	Let $M=\max\limits_{V} (u-v)=(u-v) (x_0)$. We claim that $M\le 0$. Otherwise, $M>0$. Thus, we deduce that \begin{equation*}
		\Delta (u-v) (x_0)=[e^{\xi+ u_0} (u-v)](x_0)>0.
	\end{equation*}
	By \eqref{1}, we have $\Delta(u-v)(x_0)\le 0.$ This is a contradiction. Thus, we have $u(x)\le v(x)$ on $V$. Therefore, we obtain $u\equiv v$ on $V$.
\end{proof}

\begin{lemma}\label{33}
	Suppose that $|V|>4\pi N$. Then there exist $U$ and $Z$ satisfying $U\ge Z$ such that \begin{equation*}
		\Delta U- e^{U+u_0}-(4\pi N f -1)\le 0,
	\end{equation*}
	and
	\begin{equation*}
		\Delta Z- e^{Z+u_0}-(4\pi N f -1)\ge 0.
	\end{equation*}
\end{lemma}
\begin{proof}
	By Propositions \ref{p1} and \ref{p2}, we see that $-P^{-1}$ is a compact operator. Suppose $(1+P^{-1}) U=0,$ then we deduce that $(\Delta-e^{u_0}) U=0.$ By a similar argument as Lemma \ref{32}, we obtain $U=0$. Thus, by Fredholm alternative, we deduce that there exists $U\in H^{1}(V)$ such that $$(1+P^{-1}) U=P^{-1} (4\pi N f-1).$$ Thus, $U$ satisfies $$(\Delta-e^{u_0}) U=4\pi N f- 1.$$ Therefore, we obtain \begin{equation*}
		\Delta U-e^{U+u_0}-(4\pi N f-1)\le \Delta U-e^{U+u_0}-(4\pi N f -1)+e^{u_0}(e^{U}-U)=0.
	\end{equation*} 
	Let $Z$ satisfying $log(1-\frac{4\pi N }{|V|}) -u_0 \ge Z$ be a solution of \begin{equation*}
		\Delta Z= 4\pi N f- \frac{4\pi N }{|V|}.
	\end{equation*} 
	It is easy to see that 
	\begin{equation*}
		\Delta Z-e^{Z+u_0}-(4\pi N f -1)=1-\frac{4\pi N }{|V|}-e^{Z+u_0} \ge 0.
	\end{equation*}
	\begin{equation*}
		\Delta Z-\Delta U\ge e^{Z+u_0}-e^{U+u_0}.
	\end{equation*}
	By mean value theorem, there exists $\eta$ such that
	$$\Delta Z-\Delta U\ge e^{Z+u_0}-e^{U+u_0}=e^{\eta +u_0}(Z-U). $$
	By a similar argument as Lemma \ref{32}, we deduce that $Z\le U$ on $V$.
	
	The proof is complete.
\end{proof}

\begin{lemma}\label{34}
	Suppose that $N<\frac{|V|}{4 \pi}$ and $u_0$ satisfies the equation \eqref{31}. Then the equation \eqref{3} admits a unique solution $W(x)$. 
\end{lemma}
\begin{proof}
	Choose a constant $K> \max\limits_{V}e^{u_0+U}$, define a sequence $\{w_n \}$ by an iterative scheme 
	\begin{equation}\label{8}
		\begin{aligned}
			(\Delta -K)w_{n+1}&=e^{u_0 +w_n}- K w_n+ 4\pi N f- 1, n=0,1,2,\cdots, \\
			w_0&= U.
		\end{aligned}
	\end{equation}
	
	We now prove 
	\begin{equation*}
		w_k \le U~ \text{for}~k\ge 1
	\end{equation*}
	by induction. By \eqref{8}  ,we see that 
	\begin{equation*}
		\Delta(w_1 - w_0)\ge K (w_1- w_0).		
	\end{equation*}
	By a similar argument as Lemma \ref{32}, we can show that $w_1 \le U$ on $V$. Suppose that $w_k\le U$ on $V$ for some integer $k\ge 1$, then
	\begin{equation}
		\begin{aligned}
			(\Delta- K) (w_{k+1} -U) &\ge e^{u_0+w_k}-e^{u_0+ U}+K (U-w_k) \\
			&= (K-e^{u_0 +\lambda} )( U- w_k)\\
			& \ge (K-e^{u_0 +U}) (U- w_k) \\
			& \ge 0,
		\end{aligned}
	\end{equation}
	where $w_k \le \lambda \le U$.
	By Lemma \ref{22}, we deduce that $w_{k+1}\le U$ on $V$.
	
	We next show that 
	\begin{equation*}
		w_{n+1} \le w_n \le \dots \le w_0
	\end{equation*}
	for any $n\ge 1$ by induction.

	Assume that $w_k\le w_{k-1}$ on $V$ for some integer $k\ge 1$, then we deduce that 
	\begin{equation*}
		\begin{aligned}
			\Delta(w_{k+1} - w_{k})- K(w_{k+1} - w_{k}) &=(e^{u_0+ w_k}- e^{u_0+w_{k-1}})-K (w_{k} - w_{k-1}) \\	
			&=(e^{u_0 + \eta} -K)(w_{k}-w_{k-1})\\
			&\ge 0	,
		\end{aligned}
	\end{equation*}
	where $w_{k}\le \eta \le w_{k-1}$. By Lemma \ref{22} , we get $w_{k+1} \le w_{k}$ on $V$.
	
	By Lemma \ref{33}, $Z\le U$. Suppose $Z\le w_{k}$ for some integer $k\ge 1$; then 
	\begin{equation*}
		\begin{aligned}
			(\Delta-K)(Z-w_{k+1})&\ge e^{u_0 +Z}-e^{u_0+ w_{k}}- K(Z-w_{k})\\
			&=(e^{u_0 +\xi } -K)(Z-w_k)\\
			&\ge 0,
		\end{aligned}
	\end{equation*}
	where $Z\le \xi \le w_k$. By Lemma \ref{22}, we have $Z \le w_{k+1}$ on $V$.  
	Therefore, we can define $W(x):=\lim\limits_{n\to +\infty} w_{n} (x)$. Clearly, $Z\le W \le U$ on $V$. Letting $n\to +\infty$ in \eqref{8}, we deduce that $W(x)$ satisfies \eqref{3}. By Lemma \ref{32}, $W(x)$ is the unique solution of the equation \eqref{32}.
\end{proof}

\begin{lemma}\label{35}
	There exists at most one solution of the equation \eqref{E}.
\end{lemma}
\begin{proof}
	For any two solutions $u$ and $v$ of the equation \eqref{E}, by mean value theorem, there exists $\zeta$ such that 
	\begin{equation}\label{m}
		\begin{aligned}
			\Delta(u-v) &= e^{u}-1- (e^{v}- 1) \\
			&=e^{\zeta} (u-v).
		\end{aligned}
	\end{equation}
	Let $M:=\max\limits_{V} (u-v)=(u-v) (x_0)$. We assert that $M\le 0$. Otherwise, $M>0$. Then, by \eqref{m}, we conclude that 
	$$\Delta(u-v)(x_0)>0.$$ By \eqref{21}, we conclude that $\Delta (u-v) (x_0)\le 0,$ which is a contradiction. Thus, we obtain $u\le v$ on $V$. By a similar discussion as above, we obtain $u\ge v$ on $V$. Therefore, we know that $u \equiv v$ on $V$. 
	
	The proof now is complete.
\end{proof}

\begin{proof}[Proof of Theorem \ref{t1}]
	The desired conclusion follows directly from Lemmas \ref{31}, \ref{34} and \ref{35}.
\end{proof}

\section{Concluding remarks}
Let $G=(V,E)$ be a connected finite graph. By the use of the maximum principle and the upper and lower solutions method, we show that there exists a unique solution to the Bogomol'nyi equation \eqref{E} if and only if 
\begin{equation}
	\sum_{i=1}^{k} n_{i} < \frac{|V|}{4\pi}.
\end{equation}
In particular, when $\mu(x)\equiv 1$, $|V|$ denotes the number of 
vertices of $G$. 

A natural question is to consider the equation \eqref{E} on a locally finite graph.  This interesting question deserves further research.

\end{document}